\documentclass[12pt,letterpaper]{article}

\usepackage{amssymb,amsfonts,amscd,amsthm}
\usepackage[all,arc]{xy}
\usepackage{enumerate, bbm}
\usepackage{mathrsfs}
\usepackage{tikz}
\usepackage[left=0.9in,top=0.9in,right=0.9in,bottom=0.9in]{geometry}
\usepackage{mathtools}
\usepackage{hyperref}
\usepackage{color}
\usepackage{graphicx}
\usepackage{enumitem} 
\graphicspath{ {TexImages/} }

\newtheorem{thm}{Theorem}[section]

\newtheorem{prop}[thm]{Proposition}
\newtheorem{lem}[thm]{Lemma}
\newtheorem{claim}[thm]{Claim}

\theoremstyle{definition}
\newtheorem{defn}{Definition}

\hypersetup{
	colorlinks,
	citecolor=blue,
	filecolor=blue,
	linkcolor=blue,
	urlcolor=blue,
	linktocpage
}

\setenumerate[1]{label=\thesection.\arabic*.} 
\setenumerate[2]{label*=\arabic*.} 

\newcommand{\Om}{\Omega}

\newcommand{\del}{\delta}

\newcommand{\pa}{\partial}

\newcommand{\sm}{\setminus}
\newcommand{\sub}{\subseteq}

\renewcommand{\c}[1]{\mathcal{#1}}

\newcommand{\tr}[1]{\textrm{#1}}

\newcommand{\cH}{\c{H}}

\title{On $t$-intersecting Hypergraphs with Minimum Positive Codegrees}
\author{Sam Spiro\footnote{Dept.\ of Mathematics, UCSD {\tt sspiro@ucsd.edu}. This material is based upon work supported by the National Science Foundation Graduate Research Fellowship under Grant No. DGE-1650112.}}
\date{\today}
\begin{document}
	\maketitle
	\begin{abstract}
		For a hypergraph $\mathcal{H}$, define the minimum positive codegree $\delta_i^+(\mathcal{H})$ to be the largest integer $k$ such that every $i$-set which is contained in at least one edge of $\mathcal{H}$ is contained in at least $k$ edges.  For $1\le s\le k,t$ and $t\le r$, we prove that for $n$-vertex $t$-intersecting $r$-graphs $\mathcal{H}$ with $\delta_{r-s}^+(\mathcal{H})>{k-1\choose s}$, the unique hypergraph with the maximum number of edges is the hypergraph $\mathcal{H}$ consisting of every edge which intersects a set of size $2k-2s+t$ in at least $k-s+t$ vertices provided $n$ is sufficiently large.  This generalizes work of Balogh, Lemons, and Palmer who proved this for $s=t=1$, as well as the Erd\H{o}s-Ko-Rado theorem when $k=s$.
	\end{abstract}

\section{Introduction}

We say that a hypergraph $\cH$ is an \textit{$r$-graph} if every edge $h\in \cH$ has size $r$, and we say that $\cH$ is \textit{$t$-intersecting} if $|h\cap h'|\ge t$ for any distinct $h,h'\in \cH$.  The central result concerning $t$-intersecting $r$-graphs is the famous Erd\H{o}s-Ko-Rado theorem.
\begin{thm}[\cite{erdos1961intersection}]\label{thm:EKR}
	For $t\le r$, if $\cH$ is an $n$-vertex $t$-intersecting $r$-graph with the maximum number of edges, then there exists a set $T$ of size $t$ such that $\cH$ consists of every edge containing $T$ provided $n$ is sufficiently large.
\end{thm}
There exist numerous extensions, variants, and applications of the Erd\H{o}s-Ko-Rado theorem; see for example the book and survey by Frankl and Tokushige~\cite{frankl2016invitation,frankl2018extremal} and the book by Godsil and Meagher~\cite{godsil2016erdos}.  Motivated by minimum degree variants of the Erd\H{o}s-Ko-Rado theorem, Balogh, Lemons, and Palmer~\cite{balogh2021maximum} considered a variant involving minimum positive codegrees.  

\begin{defn}
	Given a hypergraph $\cH$ and integer $i$, we define its \textit{minimum positive $i$-degree} $\del^+_i(\cH)$ to be the largest integer $k$ such that if $S$ is a set of $i$ vertices contained in at least one edge, then $S$ is contained in at least $k$ edges.  We adopt the convention that $\del^+_i(\cH)=\infty$ if $\cH$ has no edges.
\end{defn}

To state the main result of \cite{balogh2021maximum}, we require one more definition.
\begin{defn}
	We say that an $r$-graph $\cH$ is an \textit{$(a,b)$-kernel system} if there exists $X\sub V(\cH)$ with $|X|=a$ such that $h\in \cH$ if and only if $|h\cap X|\ge b$. 
\end{defn}
For example, $(t,t)$-kernel systems are exactly the extremal constructions appearing in the Erd\H{o}s-Ko-Rado theorem.  More generally, Ahlswede and Khachatrian~\cite{ahlswede1997complete} showed that every $n$-vertex $t$-intersecting $r$-graph $\cH$ with the maximum number of edges is a $(2i+t,i+t)$-kernel system for some $i$ provided $n>2r-t$.  There are many other contexts where kernel systems appear as extremal constructions for variants of the Erd\H{o}s-Ko-Rado theorem, especially for problems related to maximum degrees; see for example \cite{frankl1978intersecting,frankl1987erdos,lemons2008unbalance}.

It is not hard to show that if $r\ge k-s+t$, then an $r$-uniform $(2k-2s+t,k-s+t)$-kernel system $\cH$ is $t$-intersecting with $\del_{r-s}^+(\cH)={k\choose s}$.   The main result of Balogh, Lemons, and Palmer~\cite{balogh2021maximum} shows that when $s=t=1$, this is the unique $r$-graph with these properties which has the maximum number of edges.

\begin{thm}[\cite{balogh2021maximum}]\label{thm:old}
		Let $k\ge 1$ and let $\cH$ be an $n$-vertex 1-intersecting $r$-graph with $\del_{r-1}^+(\cH)\ge k$.  If $\cH$ has the maximum number of edges amongst hypergraphs with these properties, then $\cH$ is a $(2k-1,k)$-kernel system if $n$ is sufficiently large in terms of $r$.
\end{thm}

The proof of Theorem~\ref{thm:old} utilized the delta-system method.  Using a similar approach together with some new ideas, we extend Theorem~\ref{thm:old} to $t$-intersecting $\cH$ which have bounded positive minimum $(r-s)$-degree for essentially all values of $s$ and $t$.
\begin{thm}\label{thm:main}
	Let $k,r,s,t$ be positive integers with $s\le k,t$ and $t\le r$, and let $\cH$ be an $n$-vertex $t$-intersecting $r$-graph with $\del_{r-s}^+(\cH)>{k-1\choose s}$.  If $\cH$ has the maximum number of edges amongst hypergraphs with these properties, then $\cH$ is a $(2k-2s+t,k-s+t)$-kernel system if $n$ is sufficiently large in terms of $r$.
\end{thm}
This theorem shows a surprising phenomenon: despite only demanding $\del_{r-s}^+(\cH)>{k-1\choose s}$ in the hypothesis, the extremal constructions of Theorem~\ref{thm:main} end up having $\del_{r-s}^+(\cH)= {k\choose s}$, which is a significantly stronger condition if $s>1$.  We note that the hypothesis $\del_{r-s}^+(\cH)>{k-1\choose s}$ in Theorem~\ref{thm:main} is best possible, since if we only demanded $\del_{r-s}^+(\cH)\ge {k-1\choose s}$, then a $(2k-2-2s+t,k-1-s+t)$-kernel system would satisfy the hypothesis and have significantly more edges.

Let us briefly discuss the range of parameters in Theorem~\ref{thm:main}.  Observe that ${k-1\choose s}=0$ for all $k\le s$, so there is no loss in generality by considering $k\ge s$.  If $s>t$, then the problem is essentially trivial in view of Theorem~\ref{thm:EKR}.  This is because a $(t,t)$-kernel system will satisfy the positive minimum positive codegree condition if $n$ is sufficiently large, and $(t,t)$-kernel systems have the maximum  number of edges amongst $t$-intersecting $r$-graphs if $n$ is sufficiently large.  If one considers $r<t$, then any $t$-intersecting $r$-graph has at most one edge, so the problem becomes trivial. Thus Theorem~\ref{thm:main} covers all of the non-trivial ranges of parameters that we could consider for this problem.

\section{Proof of Theorem~\ref{thm:main}}
Our argument starts off nearly identical to that of \cite{balogh2021maximum}.  We note that Theorem~\ref{thm:main} implicitly says that if $r<k-s+t$, then any $\cH$ satisfying the hypothesis of Theorem~\ref{thm:main} is empty (since $(a,b)$-kernel systems are empty if $r<b$).  The following confirms this is the case.
\begin{lem}\label{lem:triv}
	For $1\le s\le k,t$ and $t\le r$, if $\cH$ is a non-empty $t$-intersecting $r$-graph  with $\del_{r-s}^+(\cH)>{k-1\choose s}$, then $r\ge k-s+t$.
\end{lem}
Here and throughout the text we refer to sets $I$ of size $i$ as \textit{$i$-sets}.
\begin{proof}
	The result is immediate if $k=s$, so assume $k>s$.  Assume for contradiction that $r\le k-s+t-1$ and let $h\in \cH$.  Because $k>s$, the minimum positive codegree condition implies that there is another edge $h'\ne h$ in $\cH$, and we will choose such an edge so that $|h\cap h'|$ is as small as possible.  
	
	Observe that $|h\cap h'|\ge t\ge s$ since $\cH$ is $t$-intersecting.  Let $S\sub h\cap h'$ be any $s$-set.  By the minimum positive codegree condition, the $(r-s)$-set $h'\sm S$ is contained in more than ${k-1\choose s}\ge {r-t+s\choose s}$ edges. As $h\sm (h'\sm S)$ has size at most $r-t+s$, we conclude that there exist some $s$-set $S'\not\sub h\sm (h'\sm S)$ such that $h'':=(h'\sm S)\cup S'\in \cH$.  Observe that $|h\cap h''|<|h\cap h'|$ since $h''$ was obtained from $h'$ be deleting an $s$-subset of $h$ and adding an $s$-set that was not contained entirely in $h$.  This contradicts us choosing $|h\cap h'|$ as small as possible, a contradiction.
\end{proof}

The remainder of our proof relies heavily on sunflowers.
\begin{defn}
	We say that a hypergraph $\c{F}$ is a \textit{sunflower} if there exists a set $X$ such that $h\cap h'=X$ for any distinct $h,h'\in \c{S}$.  In this case we say that $X$ is the \textit{core} of $\c{F}$ and that the sets $h\sm X$ with $h\in \c{F}$ are the \textit{petals} of $\c{F}$.  When $\c{F}$ consists of a single edge $h$, we adopt the convention that $h$ is the core of $\c{F}$.
\end{defn}

The main result in the theory of sunflowers is the following result of Erd\H{o}s and Rado.

\begin{thm}[\cite{erdosSunflower}]\label{thm:sunflower}
	For every $r,p\ge 1$, there exists a constant $f(r,p)\le r!(p-1)^r$ such that if $\cH$ is an $r$-graph with more than $f(r,p)$ edges, then $\cH$ contains a sunflower with at least $p$ petals.
\end{thm}
Much stronger bounds for $f(r,p)$ have been obtained in breakthrough work of Alweiss et. al. \cite{alweiss2020improved}, but for our purposes we only need that $f(r,p)$ is a constant. Theorem~\ref{thm:sunflower} does not give any control over the size of the core of a sunflower in $\cH$, and for this we use a result of Mubayi and Zhao~\cite{mubayi2007forbidding} which is based off of work of F\"uredi~\cite{furedi1983finite}.
\begin{prop}[\cite{mubayi2007forbidding}]\label{prop:MZ}
	If $r> k-s+t$ and $p\ge 1$, then there exists a constant $C$ depending on $r,p$ such that if $|\cH|\ge C n^{r-k+s-t-1}$, then $\cH$ contains a sunflower with at least $p$ petals and core of size at most $k-s+t$.
\end{prop}

Proposition~\ref{prop:MZ} allows us to find sunflowers which have many petals and small cores.  The next lemma shows that cores of sunflowers with many petals can not be too small.

\begin{lem}\label{lem:largeCore}
	For $1\le s\le k,t$, let $\cH$ be a $t$-intersecting $r$-graph with $\del_{r-s}^+(\cH)>{k-1\choose s}$.  If $\c{F}$ is a sunflower of $\cH$ with at least $r+1$ petals and core $Y$, then $|Y|\ge k-s+t$.
\end{lem}
\begin{proof}
	We first observe that every edge $h\in \cH$ intersects $Y$ in at least $t$ vertices.  Indeed, because $\c{F}$ has at least $r+1$ petals, there exists some petal $P$ of $\c{F}$ which is disjoint from $h$, and having $h$ and $P\cup Y$ as edges in $\cH$ implies that $|h\cap Y|\ge t$.
	
	Assume for contradiction that $|Y|<k-s+t$, and let $Z$ be a smallest set of vertices with the property that $|h\cap Z|\ge t$ for all $h\in H$.  Observe that $|Z|\le |Y|<k-s+t$ since $Y$ is a set with this property.  By the minimality of $|Z|$, there exists some $h\in \cH$ which intersects $Z$ in exactly $t$ vertices $\{z_1,\ldots,z_t\}$.  Note that $h\sm \{z_1,\ldots,z_s\}$ is an $(r-s)$-set contained in an edge.  Moreover, since every edge $h'$ intersects $Z$ in at least $t$ vertices, every edge $h'$ containing this $(r-s)$-set must be of the form $(h\sm \{z_1,\ldots,z_s\})\cup S$ with $S$ an $s$-subset of  $Z\sm \{z_{s+1},\ldots,z_t\}$ since we have $\{z_{s+1},\ldots,z_t\}\sub h\sm \{z_1,\ldots,z_s\}$.  The number of choices of such $s$-sets is exactly ${|Z|-t+s\choose s}\le {k-1\choose s}$, a contradiction to this $(r-s)$-set being contained in more than ${k-1\choose s}$ edges.  We conclude the result.
\end{proof}

At this point in the analogous proof of \cite{balogh2021maximum} for $s=t=1$ with\footnote{There is a small error in \cite{balogh2021maximum} where it is claimed that their argument works for $r\ge k$ as opposed to just $r>k$.  However, a simple modification of their argument gives a correct proof for the $r=k$ case.} $r>k$, it is argued that if $|\cH|\gg n^{r-k-1}$, then $\cH$ contains a set $Y\cup Z$ of size $2k-1$ such that every $k$-subset of $Y\cup Z$ is the core of a sunflower with at least $r+1$ petals.  From this observation one can quickly deduce that every edge intersects $Y\cup Z$ in at least $k$ vertices, which implies Theorem~\ref{thm:old}.  

Essentially this same argument will work in our general setting if one assumes the stronger hypothesis $\del_{r-s}^+(\cH)\ge {k\choose s}$, but it fails when considering the weaker hypothesis $\del_{r-s}^+(\cH)> {k-1\choose s}$.  For example, if $s=t=2$ and $k=4$, then one can consider the $r$-graph $\cH$ which consists of every edge containing at least 4 vertices of $\{1,2,3,4,5,6\}$ except for the edges which contain $\{1,2,3,4\}$.  This construction has many edges and satisfies $\del_{r-2}^+(\cH)=5>{k-1\choose s}$, but it is not the case that there is a set of size $2k-2s+t$ such that every $(k-s+t)$-subset is the core of a sunflower with many edges.  Thus from this point onwards we will have to  deviate significantly from the approach of \cite{balogh2021maximum}.  The key definition we need is the following.

\begin{defn}
	Given integers $k,s,t$ and a hypergraph $\cH$, we say that a triple of vertex sets $(h,Y,Z)$ is a \textit{bad triple} if the following conditions hold:
	\begin{itemize}
	\item[(1)] The set $h$ is an edge of $\cH$ with $|h\cap (Y\cup Z)|<k-s+t$.
	\item[(2)] We have $|Y|=|Z|=k-s+t$ and $|Y\cup Z|=2k-2s+t$ (or equivalently, $|Y\cap Z|=t$).
	\item[(3)] The sets $Y,Z$ are cores of sunflowers $\c{F}_Y,\c{F}_Z$.  Moreover, every petal $P$ of $\c{F}_Y$ is disjoint from every edge of $\c{F}_Z$, and every petal $Q$ of $\c{F}_Z$ is disjoint from every edge of $\c{F}_Y$.
	\item[(4)] For every edge $h'\in \cH$ there exist a petal $P$ of $\c{F}_Y$ and a petal $Q$ of $\c{F}_Z$ such that $h'\cap P=h'\cap Q=\emptyset$.
	\item[(5)]  For any petal $P$ of $\c{F}_Y$, define
	\[I(P)=\{Y':Y'\sub (Y\cup Z),\ P\cup Y'\in \cH\}.\]
	We have $I(P)=I(P')$ for all petals $P,P'$ of $\c{F}_Y$.
	\end{itemize}
\end{defn}
We note that if $r=k-s+t$, then all of the conditions except (1) are satisfied if there exist two edges $Y,Z$ with $|Y\cap Z|=t$ (since one can take $\c{F}_Y,\c{F}_Z$ to be sunflowers with 1 edge and the empty set as a petal).  If $r>k-s+t$, then (4) will be satisfied provided $\c{F}_Y,\c{F}_Z$ each have at least $r+1$ edges.

Condition (5) will mostly be used as a technical convenience as follows: if there exists an edge $P\cup Y'$ with $P$ a petal of $\c{F}_Y$ and $Y'\sub Y\cup Z$ a set of size $k-s+t$, then (5) guarantees that $P'\cup Y'$ will be an edge for any petal $P'$ of $\c{F}_Y$.   We also note that (5) is the only condition which is asymmetric in $Y$ and $Z$.

The following two results show that bad triples are the only obstruction to proving Theorem~\ref{thm:main}.

\begin{lem}\label{lem:dumbR}
	For $1\le s\le t,k$ and $r=k-s+t$, let $\cH$ be an $n$-vertex $t$-intersecting $r$-graph with $\del_{r-s}^+(\cH)>{k-1\choose s}$.  If $\cH$ contains no bad triples, then $\cH$ is a subset of a $(2k-2s+t,k-s+t)$-kernel system.
\end{lem}
\begin{proof}
	The result is trivial if $\cH$ is empty, so assume $\cH$ contains at least one edge.  Let $Y,Z$ be (possibly non-distinct) edges of $\cH$ such that $|Y\cap Z|$ is as small as possible.  We claim that $|Y\cap Z|=t$.  Indeed, we must have $|Y\cap Z|\ge t$ since $\cH$ is $t$-intersecting, so assume for contradiction that $|Y\cap Z|\ge t+1$.  Let $S_Z\sub Y\cap Z$ be an $s$-set.  Because \[|Y\sm (Z\sm S_Z)|\le k-s+t-(t+1-s)=k-1,\] the positive minimum codegree condition implies there is an $s$-set $\tilde{S}$ such that $\tilde{Z}:=(Z\sm S_Z)\cup \tilde{S}$ is an edge with $\tilde{S}\not\sub Y\sm (Z\sm S_Z)$, and in particular $\tilde{S}\not\sub Y$ since $\tilde{S}$ is disjoint from $Z\sm S_Z$.  Note that $|Y\cap \tilde{Z}|<|Y\cap Z|$ since $\tilde{Z}$ was formed from $Z$ by deleting an $s$-set $S_Z\sub Y$ and adding an $s$-set $\tilde{S}\not\sub Y$.  This contradicts us choosing $Z$ to minimize $|Y\cap Z|$, so we conclude that $|Y\cap Z|=t$.
	
	As noted before the lemma, if there existed an edge $h$ with $|h\cap (Y\cup Z)|<k-s+t$, then $(h,Y,Z)$ would be a bad triple.  Thus no such edge exists by hypothesis.  We conclude that $\cH$ is a subset of a $(2k-2s+t,k-s+t)$-kernel system, namely the one consisting of every edge which intersects $Y\cup Z$ in at least $k-s+t$ vertices.
\end{proof}
\begin{prop}\label{prop:stability}
	For $1\le s\le t,k$ and $r> k-s+t$, let $\cH$ be an $n$-vertex $t$-intersecting $r$-graph with $\del_{r-s}^+(\cH)>{k-1\choose s}$.  There exists a constant $C=C(r)$ such that if $|\cH|\ge C n^{r-k+s-t-1}$ and $\cH$ contains no bad triples, then $\cH$ is a subset of a $(2k-2s+t,k-s+t)$-kernel system.
\end{prop}
\begin{proof}
	Let $\cH$ be as in the proposition.  By Proposition~\ref{prop:MZ} and Lemma~\ref{lem:largeCore}, we can guarantee, provided $C$ is sufficiently large in terms of $r$, that $\cH$ contains a sunflower $\c{F}_Y'$ with core $Y=\{y_1,\ldots,y_{k-s+t}\}$ and at least \[p=(2r+1)f(s,r+1)^{k-s+t}+ (r+1)2^{2^{2k-2s+t}}+r(r+1)\] petals, where $f(s,r+1)$ is as in Theorem~\ref{thm:sunflower}.  Analogous to the proof of Lemma~\ref{lem:dumbR}, we wish to show the following.
	
	\begin{claim}\label{cl:YZ}
		There is a set of vertices $Z=\{z_1,\ldots,z_{k-s+t}\}$ such that $|Y\cap Z|=t$ and such that $Z$ is the core of a sunflower $\c{F}_Z'$ with at least $2r+1$ petals.  
	\end{claim}
	\begin{proof}
		For all $i\ge 0$, define $g(i)=(2r+1)f(s,r+1)^{i}$.  We say that a $(k-s+t)$-set $Z$ is \textit{good} if there is a sunflower $\c{F}_Z'$ with core $Z$ and at least $g(|Y\cap Z|)$ petals.  Note that there exists a good set, namely $Y$.  Let $Z$ be a good set such that $|Y\cap Z|$ is as small as possible.  We claim that $|Y\cap Z|=t$, from which the claim will follow since $Z$ is contained in at least $g(t)\ge g(0)= 2r+1$ petals.
		
		We first observe that $|Y\cap Z|\ge t$.  Indeed if this was not the case, then since $Z$ is the core of a sunflower with at least $g(0)\ge r+1$ petals, one of these petals $Q$ must be disjoint from $Y$.  Similarly one can find a petal $P$ of $\c{F}_Y'$ which is disjoint from $Q\cup Z$, which means the edges $Q\cup Z$ and $P\cup Y$ intersect in less than $t$ vertices, a contradiction.
		
		Assume for contradiction that $|Y\cap Z|\ge t+1$.  Fix an $s$-set $S_Z\sub Y\cap Z $ and observe \[|Y\sm (Z\sm S_Z)|\le (k-s+t)-(t+1-s)=k-1.\] This implies that for each petal $Q$ of $\c{F}_Z'$, there exists some $s$-set $\tilde{S}_Q$ such that $\tilde{S}_Q\not\sub Y$ and $Q\cup (Z\sm S_Z)\cup \tilde{S}_Q\in \cH$, as otherwise the $(r-s)$-set $Q\cup (Z\sm S_Z)$ would be contained in at most ${k-1\choose s}$ edges.  
		
		Consider the $s$-graph $\c{S}$ with edge set $\{\tilde{S}_Q:Q\tr{ a petal of }\c{F}_Z'\}$.  We claim that $\c{S}$ does not contain a sunflower with at least $r+1$ petals.  Indeed, assume $\c{S}$ had distinct edges $\tilde{S}_{Q_1},\ldots,\tilde{S}_{Q_{r+1}}$ which all have pairwise intersection $W$.  In this case $\cH$ would contain a sunflower $\c{F}$ with edges $Q_i\cup (Z\sm S_Z)\cup \tilde{S}_{Q_i}$ and core $(Z\sm S_Z)\cup W$.  Note that $|W|<s$ since it is the core of an $s$-uniform sunflower with more than one edge, so $\c{F}$ is a sunflower in $\cH$ with at least $r+1$ petals and a core of size less than $k-s+t$, a contradiction to Lemma~\ref{lem:largeCore}.
		
		With this we have $|\c{S}|\le f(s,r+1)$, and hence there exists some $\tilde{S}\in \c{S}$ such that $\tilde{S}_Q=\tilde{S}$ for at least $g(|Y\cap Z|)/f(s,r+1)=g(|Y\cap Z|-1)$ petals $Q$ of $\c{F}_Z'$.  Thus the set $\tilde{Z}:=(S\sm S_Z)\cup \tilde{S}$ is the core of a sunflower with at least $g(|Y\cap \tilde{Z}|-1)$ petals.    Note that $|Y\cap \tilde{Z}|<|Y\cap Z|$ because $S_Z\sub Y$ and $\tilde{S}\subsetneq Y$ by definition of $\tilde{S}_Q$, a contradiction to us choosing $Z$ to be good with $|Y\cap Z|$ as small as possible.  In total this implies $|Y\cap Z|\le t$, completing the proof.
	\end{proof}
	With $Z$ as in the claim, there are at least $r+1$ petals $Q_1,\ldots,Q_{r+1}$ which are disjoint from $Y$, and we let $\c{F}_Z$ denote the sunflower with these petals.  Amongst the petals of $\c{F}_Y'$, there are at least $(r+1)2^{2^{2k-2s+t}}$ petals which are disjoint from every edge of $\c{F}_Z$, and amongst these petals there must exist at least $r+1$ petals $P_1,\ldots,P_{r+1}$ such that $I(P_i)=I(P_j)$ for all $i,j$ (since there are at most $2^{2^{2k-2s+t}}$ possible sets that $I(P)$ could be).  Let $\c{F}_Y$ be the sunflower with core $Y$ using these $r+1$ petals.  We must have $|h\cap (Y\cup Z)|\ge k-s+t$ for all $h\in \cH$, as otherwise $(h,Y,Z)$ would be a bad triple.  This means $\cH$ is a subset of a $(2k-2s+t,k-s+t)$-kernel system, namely the one consisting of every edge which intersects $Y\cup Z$ in at least $k-s+t$ vertices.
\end{proof}
It remains to argue that hypergraphs as in Theorem~\ref{thm:main} do not have bad triples.

\begin{lem}\label{lem:lessS}
	For $1\le s\le k,t$, let $\cH$ be a $t$-intersecting $r$-graph with $\del_{r-s}^+(\cH)>{k-1\choose s}$.   If $\cH$ has a bad triple $(h,Y,Z)$ with $|h\cap Z|=t$, then $|h\cap Y\cap Z|<s$.
\end{lem}
\begin{proof}
	Assume for contradiction that there exists a bad triple $(h,Y,Z)$ as in the lemma statement and an $s$-set $S\sub h\cap Y\cap Z$.  Let $Q$ be a petal of $\c{F}_Z$ which is disjoint from $h$, and consider the $(r-s)$-set $Q\cup (Z\sm S)$.  Observe that if $h'=Q\cup (Z\sm S)\cup S'$ is an edge containing this $(r-s)$-set and $P$ is a petal of $\c{F}_Y$ disjoint from $h'$, then \[|h'\cap (P\cup Y)|= |(Z\sm S)\cap Y|+|S'\cap Y|=t-s+|S'\cap Y|,\]
	so we must have $|S'\cap Y|=s$, which implies $S'\sub Y\sm (Z\sm S):=Y'$ since $S'$ must be disjoint from $Z\sm S$.  As there are more than ${k-1\choose s}$ edges containing $Q\cup (Z\sm S)$, and since $|Y'|=k-s+t-(t-s)=k$, we conclude that for every $y\in Y'$ there exists an edge $h_y$ containing $y$ and $Q\cup (Z\sm S)$.  Because $(h,Y,Z)$ is a bad triple, we have $|h\cap (Y\cup Z)|<k-s+t=|Y|$, and hence there exists some $y\in Y\sm h$.  With this we have $|h \cap h_y|\le t-1$ because $|h\cap (Q\cup Z)|=t$ and we construct $h_y$ by deleting from $Q\cup Z$ an $s$-set $S\sub h\cap Z$ and then adding an $s$-set using some $y\notin h$.  This contradicts $\cH$ being $t$-intersecting, giving the result.
\end{proof}

\begin{lem}\label{lem:tInt}
	For $1\le s\le k,t$, let $\cH$ be a $t$-intersecting $r$-graph with $\del_{r-s}^+(\cH)>{k-1\choose s}$.  If $\cH$ has a bad triple, then it has a bad triple of the form $(h,Y,Z)$ with $|h\cap Z|=t$ and $|h\cap Y\cap Z|\ge t-s+1$.
\end{lem}
\begin{proof}
	Let $(h,Y,Z)$ be a bad triple with $|h\cap Z|$ as small as possible, and assume for contradiction that $|h\cap Z|\ge t+1$.  Let $S\sub h\cap Z$ be an arbitrary set of $s$ vertices.  Observe that if $h'$ is an edge containing $h\sm S$ and $S':=h'\sm (h\sm S)$, then $S'\sub Z\sm (h\sm S)$, as otherwise $(h',Y,Z)$ would be a bad triple with $|h'\cap Z|<|h\cap Z|$, contradicting the minimality of $(h,Y,Z)$.  We have $|Z\sm (h\sm S)|\le (k-s+t)-(t+1-s)=k-1$, so there are at most ${k-1\choose s}$ edges containing $h\sm S$, a contradiction to the minimum positive codegree condition.  We conclude that $|h\cap Z|\le t$, and we must have $|h\cap Z|\ge t$ since $\cH$ is $t$-intersecting and $\c{F}_Z$ contains a petal which is disjoint from $h$. We conclude that $|h\cap Z|=t$. 
	
	Now consider $(h,Y,Z)$ a bad triple with $|h\cap Z|=t$ and with $|h\cap Y\cap Z|$ as large as possible.  Asume for contradiction that $|h\cap Y\cap Z|\le t-s$, which implies there exists an $s$-set $S\sub (Y\cap Z)\sm h$  and also some $z\in (Z\cap h)\sm Y$.  Let $P$ be a petal of $\c{F}_Y$, and observe that the $(r-s)$-set $P\cup (Y\sm S)$ intersects any edge $Q\cup Z$ with $Q$ a petal of $\c{F}_Z$ in $t-s$ vertices.  Thus any edge containing $P\cup (Y\sm S)$ must also contain an $s$-set from the set $Z\sm (Y\sm S)$.  As $|Z\sm (Y\sm S)|=k$, and since the $(r-s)$-set is in more than ${k-1\choose s}$ edges, there must exist a set $S'\sub Z\sm (Y\sm S)$ such that $P\cup (Y\sm S)\cup S'$ is an edge with $z\in S'$.  Define $Y'=(Y\sm S)\cup S'$, noting that $|Y'\cap Z|=t$.  Observe that $|h\cap Y'\cap Z|>|h\cap Y\cap Z|$, since we formed $Y'$ by deleting from $Y$ an $s$-set which is disjoint from $h$ and then added an $s$-set containing a vertex in $Z\cap h$.  Also observe that due to condition (5) for bad triples, $Y'$ is the core of a sunflower whose petals are the same as $\c{F}_Y$ (i.e. since $P\cup Y'$ is an edge and $Y'\sub Y\cup Z$, we have that $P'\cup Y'$ is an edge for every petal $P'$ of $\c{F}_Y$).  Further, $Y'\cup Z=Y\cup Z$, so (5) continues to hold and we conclude that $(h,Y',Z)$ is a bad triple with $|h\cap Z|=t$ and $|h\cap Y'\cap Z|>|h\cap Y\cap Z|$, a contradiction to our choice of triple.
\end{proof}

Note that these two lemmas immediately show that there are no bad triples if $t\ge 2s-1$, but we will need to work harder to get the result for all $t$.  The last tool we need is a particular case of the Kruskal-Katona theorem.

\begin{lem}[\cite{katona2009theorem,kruskal1963number}]\label{lem:KK}
	For an $s$-graph $\c{S}$ and $0\le i\le s$, let $\pa^i \c{S}$ be the $i$-graph with edge set $\{S':S'\sub S\in \cH,\ |S'|=i\}$.  If $|\c{S}|={k-1\choose s}$, then $|\pa^i \c{S}|\ge {k-1\choose i}$.
\end{lem}

With this we can prove that bad triples do not exist.

\begin{prop}\label{prop:noBad}
	For $1\le s\le k,t$, if $\cH$ is a $t$-intersecting $r$-graph with $\del_{r-s}^+(\cH)>{k-1\choose s}$, then $\cH$ does not contain a bad triple.
\end{prop}
\begin{proof}
	Assume for contradiction that $\cH$ contains a bad triple, so by Lemma~\ref{lem:tInt} there exists a bad triple $(h,Y,Z)$ with $|h\cap Z|=t$ and $|h\cap Y\cap Z|\ge t-s+1$.  Let $T\sub h\cap Y\cap Z$ be a set of $t-s$ vertices, and let $S_h=(h\cap Z)\sm T$ and $S_Y=(Y\cap Z)\sm T$.  Note that $S_h,S_Y$ are $s$-sets since $h,Y$ both intersect $Z$ in exactly $t$ vertices.

	\begin{claim}
		Let $P$ be a petal of $\c{F}_Y$ which is disjoint from $h$.  Define \[\c{S}_h=\{S: |S|=s,\ (h\sm S_h)\cup S\in \cH\},\hspace{1em} \c{S}_Y=\{S: |S|=s,\ P\cup (Y\sm S_Y)\cup S\in \cH\}.\]  The set systems $\c{S}_h,\c{S}_Y$ have the following properties:
		\begin{itemize}
			\item[(a)] We have $|\c{S}_h|,|\c{S}_Y|>{k-1\choose s}$,
			\item[(b)] Every $S\in \c{S}_h\cup \c{S}_Y$ is an $s$-set which is a subset of the $k$-set $Z\sm T$,
			\item[(c)] The sets $\c{S}_h,\c{S}_Y$ are disjoint, and
			\item[(d)] Every $S'_h\in \c{S}_h$ and $S'_Y\in \c{S}_Y$ satisfy $|S'_h\cap S'_Y|\ge 2s+1-k$.
		\end{itemize}
	\end{claim}
	\begin{proof}
		Property (a) follows immediately since the $(r-s)$-sets $h\sm S_h$ and $P\cup (Y\sm S_Y)$ are contained in an edge.  For (b), every element of $\c{S}_h\cup \c{S}_Y$ is an $s$-set by definition.  Note that the $(r-s)$-set $h\sm S_h$  intersects $Z$ in $t-s$ vertices.  If $S$ is such that $(h\sm S_h)\cup S$ is an edge, then one can choose a petal $Q$ of $\c{F}_Z$ which is disjoint from $(h\sm S_h)\cup S$, which means \[|((h\sm S_h)\cup S)\cap (Q\cup Z)|=t-s+|S\cap Z|.\]  This implies that we must have $|S\cap Z|\ge s$, i.e. $S\sub Z$.   Moreover we have $S\sub Z\sm T$ since $S$ must be disjoint from $(h\sm S_h)\supseteq T$.   An identical argument holds for $\c{S}_Y$, proving (b).
		
		For (c), assume for contradiction that there existed $S\in \c{S}_h\cap \c{S}_Y$.  Let $h'=(h\sm S_h)\cup S$ and $Y':=(Y\sm S_Y)\cup S$.  Because $P\cup Y'$ is an edge, condition (5) for bad triples implies $Y'\sub Y\cup Z$ is the core of a sunflower using the same petals as $\c{F}_Y$.  Thus $(h',Y',Z)$ is a bad triple with $|h'\cap Z|=t$ and $|h'\cap Y'\cap Z|\ge |S|=s$, a contradiction to Lemma~\ref{lem:lessS}.
		
		For (d), using $S_h,S_Y\sub Z$ together with $|h\cap Z|=t$ and $|h\cap (Y\cup Z)|\le k-s+t-1$ implies \[|((h\sm S_h)\cap (Y\sm S_Y))\sm Z|=|(h\cap Y)\sm Z|\le k-s-1.\]  We also have $|(h\sm S_h)\cap (Y\sm S_Y)\cap  Z|=|T|=t-s$, so in total \[|(h\sm S_h)\cap (Y\sm S_Y)|\le k+t-2s-1.\]  
		Observe that if $S_h'\in \c{S}_h$, then $S_h'$ is disjoint from $P\cup (Y\sm S_Y)$ since $S_h'\sub Z\sm T$.  Similarly every $S_Y'\in \c{S}_Y$ is disjoint from $h\sm S_h$. 	Thus if $S_h'\in \c{S}_h$ and $S'_Y\in \c{S}_Y$, for their corresponding edges to intersect in at least $t$ vertices we must have $|S'_h\cap S'_Y|\ge t-(k+t-2s-1)=2s+1-k$, proving the result.
	\end{proof}  
	It remains to show that it is impossible to construct set systems $\c{S}_h,\c{S}_Y$ with the properties as in this claim.  Observe that properties (b), (c), (a) imply that
	\[{k\choose s}\ge |\c{S}_h\cup \c{S}_Y|=|\c{S}_h|+|\c{S}_Y|\ge 2+2{k-1\choose s}.\]
	Note that this inequality is equivalent to ${k-1\choose s-1}\ge 2+{k-1\choose s}$, which is false for $k\ge 2s$.  Thus from now on we may assume $k\le 2s-1$, which in particular means (d) is a non-trivial condition.
	
	Observe that $s\le k\le 2s-1$ implies $0\le k-s\le s-1$, so $\pa^{k-s}\c{S}_h$ is well defined.  Let $\c{S}'$ be the $s$-graph with edge set $\{(Z\sm T)\sm S:S\in \pa^{k-s}\c{S}_h\}$.  That is, $\c{S}'$ is the ``complement'' of $\pa^{k-s}\c{S}_h$ in $Z\sm T$. For every $S'\in \c{S}'$, there exists some $S\in \c{S}_h$ such that $|S\cap S'|=2s-k$; namely, this holds whenever $S'$ is the complement of a $(k-s)$-subset of $S$.  By (d), it must be that $\c{S}'$ and $\c{S}_Y$ are disjoint $s$-graphs on $Z\sm T$.  Using this together with (a), $|\c{S}'|=|\pa^{k-s}\c{S}_h|$, and Lemma~\ref{lem:KK} gives
	\[{k\choose s}\ge |\c{S}_Y\cup \c{S}'|=|\c{S}_Y|+|\c{S}'|>{k-1\choose s}+{k-1\choose k-s}={k-1\choose s}+{k-1\choose s-1}={k\choose s}.\]
	This is a contradiction, proving the result.
\end{proof}

\begin{proof}[Proof of Theorem~\ref{thm:main}]
	Let $\cH$ be a hypergraph satisfying the hypothesis of the theorem with the maximum number of edges, and observe that $\cH$ has no bad triples by Proposition~\ref{prop:noBad}.  If $r<k-s+t$, then $\cH$ is empty by Lemma~\ref{lem:triv}.  If $r=k-s+t$, then $\cH$ is contained in a $(2k-2s+t,k-s+t)$-kernel system by Lemma~\ref{lem:dumbR}, and since $\cH$ has the maximum number of edges it must in fact be such a kernel system.  Thus we may assume $r> k-s+t$. 
	
	 Because $(2k-2s+t,k-s+t)$-kernel systems satisfy the hypothesis of the theorem and have $\Om(n^{r-k+s-t})$ edges, we must have $|\cH|=\Om(n^{r-k+s-t})$ as well.  Proposition~\ref{prop:stability} implies $\cH$ is contained in a $(2k-2s+t,k-s+t)$-kernel system if $n$ is sufficiently large, and because $\cH$ has the maximum number of edges, it must be such a kernel system, proving the result.
\end{proof}

Our proof in fact gives a stability result: if $\cH$ is as in Theorem~\ref{thm:main} with $|\cH|\gg n^{r-k+s-t-1}$, then $\cH$ is a subset of a $(2k-2s+t,k-s+t)$-kernel system.  In the special case of $r=k-s+t$, this conclusion holds regardless of the size of $\cH$.  

We made no attempt at optimizing how large $n$ must be for Theorem~\ref{thm:main} to hold.  A careful analysis shows that our argument works provided $n\approx r^{rs(k-s+t)}+2^{r2^{2k-2s+t}}$ due to demanding a sunflower with sufficiently many petals to find a sunflower on $Z$, as well as to utilize condition (5) for bad triples.   It is not too difficult to provide an alternative proof by replacing (5) with the condition that if $r>k-s+t$, then the sunflower $\c{F}_Y$ contains at least $(2r)^{k(s+1-|h\cap Y\cap Z|)}$ petals.  This ultimately gives a bound that works for $n\approx r^{rs(k-s+t)}$, which matches the bound of $n\approx r^{rk}$ implicitly proven in \cite{balogh2021maximum} when $s=t=1$.  However, this alternative proof is slightly more involved, so we elected to use the current argument at the cost of weaker (implicit) bounds.

As in \cite{balogh2021maximum}, one may be able to prove that Theorem~\ref{thm:main} holds for $n\approx r^{s(k-s+t)}$ if $k$ is small, and in particular one may be able to adapt the ideas of \cite{balogh2021maximum} to prove such a bound when $k=s+1$.

\textbf{Acknowledgments}.  The author thanks Jason O'Neill for engaging conversations and comments on an earlier draft, and Cory Palmer for clarifying some of the points from \cite{balogh2021maximum}.

	\bibliographystyle{abbrv}
	\bibliography{Codegree}
\end{document}